\documentclass[11pt,oneside,english,a4paper]{amsart}
\usepackage[T1]{fontenc}
\usepackage[utf8]{inputenc}
\usepackage[dvipsnames]{xcolor}
\usepackage[yyyymmdd,hhmmss]{datetime}
\usepackage[numbers,sort,compress]{natbib}
\usepackage{amstext,amsthm,amssymb}
\usepackage{xparse,mathrsfs,mathtools}
\usepackage[english]{babel}
\usepackage[unicode=true,%
	pdfusetitle,pdfdisplaydoctitle=true,%
	pdfpagemode=UseOutlines,%
	bookmarks=true, bookmarksnumbered=false, bookmarksopen=true, bookmarksopenlevel=2,%
	breaklinks=true,%
	pdfencoding=unicode,psdextra,
	pdfcreator={},
	pdfborder={0 0 0}
]{hyperref}
\usepackage{lmodern,microtype}
\usepackage[a4paper]{geometry}
\usepackage[capitalise,nameinlink]{cleveref}
\usepackage[cal=dutchcal,bb=ams,scr=boondoxo]{mathalfa}

\hypersetup{
	colorlinks=false,
	pdfborder={0 0 0},
	pdfborderstyle={/S/U/W 0},
}

\newcommand*{\doi}[1]{\href{http://dx.doi.org/#1}{doi: {\texttt{\footnotesize\color{teal}#1}}}}
\newcommand*{\cUrl}[1]{\href{#1}{{\texttt{\footnotesize\color{teal}#1}}}}

\newtheorem*{thm*}{Theorem}
\newtheorem{thm}{Theorem}
\newtheorem{lem}[thm]{Lemma}
\newtheorem{prop}[thm]{Proposition}
\newtheorem*{conj*}{Conjecture}
\theoremstyle{remark}
\newtheorem{rem}[thm]{Remark}
\newtheorem{example}[thm]{Example}
\newtheorem*{rems*}{Remarks}
\newtheorem*{rem*}{Remark}

\newtheorem{remappendix}{Remark}[section]

\newcommand{\ZZ}{\mathbb{Z}}
\newcommand{\QQ}{\mathbb{Q}}

\newcommand{\PGL}{\operatorname{PGL}}

\DeclarePairedDelimiter\parentheses{\lparen}{\rparen}
\DeclarePairedDelimiter\braces{\lbrace}{\rbrace}
\NewDocumentCommand\set{ s o m o }{%
	\IfBooleanTF{#1}{\IfNoValueTF{#4}{\braces*{#3}}{\braces*{\,#3:#4\,}}}{%
	\IfNoValueTF{#2}{\IfNoValueTF{#4}{\braces{#3}}{\braces{\,#3:#4\,}}}{%
	\IfNoValueTF{#4}{\braces[#2]{#3}}{\braces[#2]{\,#3:#4\,}}}}%
}
\newcommand{\sym}{\mathfrak{S}}

\crefname{section}{§}{§§}
\crefname{figure}{Figure}{Figures}
\crefformat{equation}{#2(#1)#3}
\crefformat{enumi}{#2(#1)#3}
\numberwithin{equation}{section}
\title[%
	Galois groups of \texorpdfstring{$\binom{\MakeLowercase{n}}{0} + \binom{\MakeLowercase{n}}{1} X + \ldots + \binom{\MakeLowercase{n}}{6} X^6$}{a truncated binomial expansion}%
]{%
	Galois groups of $\displaystyle \binom{n}{0} + \binom{n}{1} X + \ldots + \binom{n}{6} X^6$%
}
\date{\today{}}

\subjclass[2020]{
	Primary
	11R09; 
	Secondary
	11R32.
}
\keywords{Truncated binomial expansion, Galois group, $\PGL(2;5)$, symmetric group}

\author{Benjamin~Klahn}
\author{Marc~Technau}
\address{%
	Benjamin~Klahn and Marc~Technau\\%
	Institut für Analysis und Zahlentheorie\\%
	Graz University of Technology\\%
	Kopernikusgasse~24/II\\%
	8010~Graz\\%
	Austria}
\email{klahn@math.tugraz.at}
\email{mtechnau@math.tugraz.at}

\begin{document}
\begin{abstract}
	We show that the Galois group of the polynomial in the title is isomorphic to the full symmetric group on six symbols for all but finitely many $n$.
	This complements earlier work of Filaseta and Moy, who studied Galois groups of $\binom{n}{0} + \binom{n}{1} X + \ldots + \binom{n}{k} X^k$ for more general pairs $(n,k)$, but had to admit a possibly infinite exceptional set specifically for $k=6$ of at most logarithmic growth in $n$.
	The proof rests upon invoking Faltings' theorem on a suitable fibration of Galois resolvents to show that this exceptional set is, in fact, finite.
\end{abstract}
\maketitle

\section{Introduction and main result}

In~\cite{filaseta2007}, Filaseta, Kumchev, and Pasechnik investigated irreducibility of the polynomials
\[
	P_{n,k} = \binom{n}{0} + \binom{n}{1} X + \ldots + \binom{n}{k} X^k
	\quad (1\leq k \leq n),
\]
whilst attributing the origin of the problem to a 2004 MSRI programme on \emph{`topological aspects of real algebraic geometry'}, in connexion with the work of Scherbak (cf.~\cite{scherbak2004}).
While the cases of $P_{n,n}= (X+1)^n$ and $P_{n,n-1}= (X+1)^n - X^n$ are easily addressed (with the latter polynomial being irreducible if and only if $n$ is prime), the situation for $1\leq k \leq n-2$ is more subtle.
The authors of~\cite{filaseta2007} mention that numerical computations seem to suggest that all such $P_{n,k}$ ought to be irreducible.
In support of this they prove, amongst other things, that for \emph{fixed} $k\geq 1$, the polynomials $P_{n,k}$ are irreducible for all \emph{sufficiently large} $n$.
(Here `sufficiently large' may depend on $k$.)
Since then, further progress has been made on the question as for which pairs $(n,k)$ the irreducibility of $P_{n,k}$ can be guaranteed; see~\cite{khanduja2011,dubickas2017,jakhar2018}.

It is natural to also ask for determining the Galois groups of the polynomials $P_{n,k}$.
This problem was studied by Filaseta and Moy~\cite{filaseta2018}.
The present note is concerned with a curious anomaly appearing in their work with respect to the Galois groups of $P_{n,6}$ as $n=8,9,\ldots$; their full result may be enunciated as follows:
\begin{thm}[Filaseta--Moy]\label{thm:FM}
	Let $k\geq 2$ be a fixed integer.
	\item
	\begin{enumerate}
		\item\label{enum:FM:1}
		For $k\neq 6$ and $n$ sufficiently large in terms of $k$, the Galois group of $P_{n,k}$ is as large as one could expect; i.e., it is isomorphic to the full symmetric group $\sym_k$ on $k$ symbols.
		\item\label{enum:FM:2}
		For $k = 6$, the set of $n=8,\ldots,N$ for which the Galois group of $P_{n,k}$ is \emph{not} isomorphic to $\sym_6$ contains \emph{at most} $O(\log N)$ many elements.
		For all of these $n$, the Galois group in question is isomorphic to $\PGL(2;5)$.
	\end{enumerate}
\end{thm}

\begin{rems*}
	\begin{enumerate}
		\item The group $\PGL(2;5)$ may be viewed as a transitive subgroup of $\sym_6$ by means of its action on the one-dimensional subspaces of $\mathbb{F}_5^2$, after fixing an arbitrary ordering of these subspaces.
		Moreover, $\PGL(2;5) \cong \sym_5$, as abstract groups.
		\item The irreducibility of $P_{n,6}$ is known for all $n\geq 8$ by the work of Dubickas and {\v{S}iurys}~\cite{dubickas2017}.
		\item The exceptional set in \cref{thm:FM}~\cref{enum:FM:2} is non-empty.
		Indeed, one can show that the Galois group of $P_{10,6}$ is, in fact, isomorphic to $\PGL(2;5)$.
	\end{enumerate}
\end{rems*}

Based on numerical computations for $n$ up to $10^{10}$, Filaseta and Moy~\cite{filaseta2018} made the following conjecture:
\begin{conj*}
	The Galois group of $P_{n,6}$ is isomorphic to $\sym_6$ for all $n\geq 11$.
\end{conj*}
We make progress towards this conjecture by reducing the logarithmic bound in \cref{thm:FM}~\cref{enum:FM:2} to some unspecified constant:
\begin{thm}\label{thm:Finiteness}
	The Galois group of $P_{n,6}$ is isomorphic to $\sym_6$ for all but finitely many positive integers $n$.
\end{thm}

The main tool in~\cite{filaseta2018} are Newton polygons, which provide information on the Galois group of integer polynomials viewed as polynomials over the field $\QQ_p$ of $p$-adic integers.
This is supplemented by information on the discriminant of $P_{k,n}$.

In order to prove \cref{thm:Finiteness}, we construct a \emph{Galois resolvent} for $\PGL(2;5)$ and $P_{n,6}$.
Roughly speaking, this is a polynomial in the coefficients of $P_{n,6}$ and an additional variable $Y$ whose admittance of a rational root $y$ (or lack thereof) serves as a detector for the Galois group of $P_{n,6}$ being contained in $\PGL(2;5)$ (up to conjugation, under suitable identifications).
(Incidentally, the aforementioned use of discriminants in~\cite{filaseta2018} can be viewed as the use of the \emph{quadratic resolvent} for the alternating group $A_n$ in $\sym_n$.)
As the coefficients of $P_{n,6}$ are polynomials in $n$ for $n\geq7$, the zero locus of the resolvent in the variables $n$ and $Y$ defines a curve which turns out to have genus $\geq 2$.
We use Faltings' finiteness theorem to show that the number of rational points on this zero locus is finite.

Both the determination of the aforementioned curve as well as its genus are computationally expensive and were accomplished using computer algebra software.
In~\cref{sec:AuxiliaryResult} below, we describe the details of the construction.
In~\cref{sec:ProofOfTheorem}, we carry out the proof of \cref{thm:Finiteness}.

\begin{rem*}
	After submitting this work, Michael Filaseta kindly informed us that \cref{thm:Finiteness} was previously obtained, using a very similar approach, by Jeremy Rouse and Jacob Juillerat (see \cite[Theorem~1.3]{juillerat2021}).
\end{rem*}

\section{Auxiliary results}
\label{sec:AuxiliaryResult}

\subsection{Galois resolvents}

Galois resolvents are a classical tool for determining Galois groups of polynomials.
Roughly speaking, given a polynomial $P$, a \emph{(Galois) resolvent} for $P$ is a polynomial whose splitting behaviour carries information about the Galois group of $P$.
In this subsection, we recall the basic results regarding resolvents that we shall require for the proof of \cref{thm:Finiteness}.
Nothing in this subsection is particularly novel.

\medskip

The resolvents we require are produced by means of the following lemma:

\begin{lem}\label{lem:resolvent}
	Let $U$ be a subgroup of the symmetric group $\sym_k$ for $k\geq 1$.
	Let $\nu_1,\ldots,\nu_k$ be non-negative integers.
	Consider the polynomial
	\[
		\Phi = \prod_{\sigma U \in \sym_k/U} \parentheses[\Big]{
			Y - \sum_{\pi \in \sigma U} \prod_{j=1}^k X_{\pi(j)}^{\nu_j}
		} \in \ZZ[Y,X_1,\ldots,X_k]
	\]
	in the variables $Y,X_1,\ldots,X_k$, where the outer-most product ranges over the left cosets $\sigma U$ of $U$ in $\sym_k$.
	Suppose that $P\in\ZZ[X]$ is a monic polynomial of degree $k$ with distinct roots $x_1,\ldots,x_k$ in the algebraic closure of $\QQ$.
	View the Galois group $G$ of $P$ as a subgroup of $\sym_k$ by means of how it permutes $x_1,\ldots,x_k$.
	Then the following assertions hold:
	\item
	\begin{enumerate}
		\item\label{enum:resolvent:symm}
		$\Phi$ can be written as an integer polynomial in $Y$ and in the elementary symmetric polynomials in the variables $X_1,\ldots, X_k$.
		\item\label{enum:resolvent:integerroot}
		If $G \leq \sym_k$ is conjugate to a subgroup of $U$, then $\Phi(Y,\boldsymbol{x})$ admits a root in the integers.
	\end{enumerate}
\end{lem}
\begin{proof}
	The special case of the present lemma with exponents $\nu_j = j$ can be found in \cite[Lemma~5]{dietmann2012}; the case with general exponents is proved in exactly the same fashion, so we omit the details.
	(We note that similar variants have also proved useful in~\cite{castillo2017,chow2023}.)
\end{proof}

\begin{rem*}
	Under some extra separability assumption, the converse to part~\cref{enum:resolvent:integerroot} of \cref{lem:resolvent} holds as well (see~\cite{lefton1977}).
	However, we do not require this fact here.
\end{rem*}

As the order of $\PGL(2;5)$ is quite large, the resolvents furnished by \cref{lem:resolvent} turn out to be quite unwieldy.
For that reason, below, we illustrate \cref{lem:resolvent} on a toy example.
The example itself serves no other purpose here, though; before giving it, we need to fix some notation.
The $j$th \emph{elementary symmetric polynomial} $e_j$ in $k$ variables $X_1,\ldots,X_k$ is given by
\[
	e_j = \sum_{\substack{
		\mathscr{M}\subseteq\set{1,\ldots,k} \\
		\#\mathscr{M} = j
	}}
	\prod_{
		m\in\mathscr{M}
		\vphantom{\subseteq\set{1,\ldots,k}}
	} X_m \in \ZZ[X_1,\ldots,X_k],
\]
where $\mathscr{M}$ ranges over all subsets of $\set{1,\ldots,k}$ with precisely $j$ elements.
In particular, one has
\begin{equation}\label{eq:ElSymmPolysExpansion}
	\prod_{j=1}^k (Y - X_j) = Y^k + \sum_{j=1}^k (-1)^j e_j Y^{k-j-1}.
\end{equation}

\begin{example}\label{example:CubicResolvent}
	We consider the case $k=3$ with the symmetric group $\sym_3$ and its subgroup $U = A_3$ (the alternating group on $3$ symbols).
	We choose $(\nu_1,\nu_2,\nu_3) = (1, 2, 0)$ in \cref{lem:resolvent}.
	Then
	\begin{equation}\label{eq:ResolventExample}
		\begin{aligned}
			\Phi &
			= \parentheses[\big]{ Y - \parentheses{ X_1^2 X_2 + X_1 X_3^2 + X_2^2 X_3 } } \mkern 1mu
			  \parentheses[\big]{ Y - \parentheses{ X_1^2 X_3 + X_1 X_2^2 + X_2 X_3^2 } } \\ &
			= (e_1^3 e_3 - 6 e_1 e_2 e_3 + e_2^3 + 9 e_3^2) + (3 e_3 - e_1 e_2) Y + Y^2,
		\end{aligned}
	\end{equation}
	where the expansion with elementary symmetric polynomials was found using a computer algebra system.
	
	Consider the polynomials $P_+ = X^3 + 3 X^2 + 3$ and $P_- = X^3 + 3 X^2 - 3$.
	Both are clearly irreducible by the Schönemann--Eisenstein criterion.
	Let $\boldsymbol{x}_\pm$ denote a triple consisting of the roots of $P_\pm$ in the algebraic closure of $\QQ$ in any order.
	We can then compute $\Phi(Y,\boldsymbol{x}_\pm)$ from the above representation with elementary symmetric polynomials, for by~\cref{eq:ElSymmPolysExpansion} we have
	\[
		e_1(\boldsymbol{x}_\pm) =  - 3, \quad
		e_2(\boldsymbol{x}_\pm) =    0, \quad\text{and}\quad
		e_3(\boldsymbol{x}_\pm) = \mp3.
	\]
	Hence,
	\[
		\Phi(Y,\boldsymbol{x}_+) = Y^2 - 9Y + 162 \quad\text{and}\quad
		\Phi(Y,\boldsymbol{x}_-) = Y^2 + 9Y.
	\]
	The first polynomial is easily seen to have no integer roots, whereas the second polynomial has the roots $0$ and $-9$.
	Consequently, the Galois group $G_+$ of $P_+$ (when viewed as a subgroup of $\sym_3$ as described in \cref{lem:resolvent}) is not conjugate to $U$, whereas the the Galois group $G_-$ of $P_-$ is.
	Using the fact that the Galois group of an irreducible cubic polynomial over $\QQ$ is isomorphic to either $\sym_3$ or $A_3$, we conclude that $G_+ \cong \sym_3$ and $G_- \cong A_3$.
\end{example}

\begin{rem}
	The choice $(\nu_1,\nu_2,\nu_3) = (1,2,3)$ produces the resolvent
	\begin{align*}
		\MoveEqLeft
		\parentheses[\big]{ Y - \parentheses{ X_1^3 X_2^2 X_3 + X_1^2 X_2 X_3^3 + X_1 X_2^3 X_3^2 } } \mkern 1mu
		\parentheses[\big]{ Y - \parentheses{ X_1^3 X_2 X_3^2 + X_1^2 X_2^3 X_3 + X_1 X_2^2 X_3^3 } } \\ &
		= (e_1^3 e_3 - 6 e_1 e_2 e_3 + e_2^3 + 9 e_3^2) e_3^2 + (3 e_3 - e_1 e_2) e_3 Y + Y^2,
	\end{align*}
	which is somewhat more complicated than the one given in~\cref{eq:ResolventExample}.
	This may serve as an explanation as to why varying the exponents $\nu_j$ may be beneficial.
	In this regard, one may also note that choosing all exponents $\nu_j$ equal to each other (say, $\nu_1=\ldots=\nu_k=\nu$) is useless.
	Indeed, in this case, the resolvent factors as $(Y - (\#U)e_k^\nu)^{[G:U]}$.
	This \emph{always} has an integer root, when specialised using the coefficients of a monic integer polynomial, and therefore cannot be used to establish that $G$ is not conjugate to $U$ in the setting of \cref{lem:resolvent}~\cref{enum:resolvent:integerroot}.
\end{rem}

\begin{rem}
	The discriminant of $\Phi$ in \cref{eq:ResolventExample} with respect to $Y$ can be seen to equal the well-known discriminant of the cubic polynomial $X^3 - e_1 X^2 + e_2 X - e_3$ (in $X$).
\end{rem}

\subsection{Constructing a resolvent for \texorpdfstring{$\boldsymbol{\PGL(2;5)}$}{PGL(2;5)}}
\label{subsec:computing:PGL25:resolvent}

As is evident from \cref{example:CubicResolvent}, the resolvents furnished by \cref{lem:resolvent} already fill a line when working with groups of very small order.
In particular, the computations for $\PGL(2;5)$ are rather unwieldy and are better left to computers.

It is known that $\PGL(2;5)$ is isomorphic to the subgroup of $\sym_6$ generated by the two permutations $(3\;6\;5\;4)$ and $(1\;2\;5) \; (3\;4\;6)$ given in cycle notation.
We choose $(\nu_1,\ldots,\nu_6) = (1,2,2,3,3,4)$ in \cref{lem:resolvent} and compute the corresponding $\Phi$.
Our particular choice of $\nu_1,\ldots,\nu_6$ was found by trial and error.
Some other choices of exponents yielded variants of $\Phi$ for which an expansion in terms of elementary symmetric polynomials seemed computationally infeasible.

\subsection{Application to the polynomials \texorpdfstring{$\boldsymbol{P_{n,6}}$}{Pₙ,₆}}
\label{sec:subsitution}

We now wish to apply the resolvent $\Phi$ computed in \cref{subsec:computing:PGL25:resolvent} to the polynomials $P_{n,6}$.

Suppose that $n\geq 8$.
First, note that the polynomial $P_{n,6}$ is not monic.
To rectify this, we consider instead the \emph{reciprocal polynomial} $P_{n,6}^{\text{rec}}$ of $P_{n,6}$,
\begin{equation}\label{eq:ReciprocalPolynomial}
	P_{n,6}^{\text{rec}}
	= X^6 P_{n,6}(X^{-1})
	= \binom{n}{6} + \ldots + \binom{n}{1} X^5 + \binom{n}{0} X^6.
\end{equation}
Note that the polynomial $P_{n,6}^{\text{rec}}$ is monic and irreducible for $n\geq 8$ (the latter follows from the irreducibility of $P_{n,6}$ for $n\geq 8$~\cite{dubickas2017}).

Let $x_1,\ldots,x_6$ denote the roots of $P_{n,6}^{\text{rec}}$ in the algebraic closure of $\QQ$ in some order.
We then compute $\Phi(Y,x_1,\ldots,x_6)$ for the $\PGL(2;5)$-resolvent $\Phi$ computed in \cref{subsec:computing:PGL25:resolvent} in the same fashion as in \cref{example:CubicResolvent}.
Observe that every binomial coefficient occurring in~\cref{eq:ReciprocalPolynomial} equals
\[
	(j!)^{-1} n (n-1) \cdots (n-j+1)
\]
for suitable $0\leq j \leq 6$; these are polynomials in $n$.
Therefore, we may view $\Phi(Y,x_1,\ldots,x_6)$ as a polynomial in $Y$ and $n$, which we shall call $P_*$:
\begin{equation}\label{eq:Resolvent}
	P_* = \Phi(Y,x_1,\ldots,x_6) \in \QQ[Y, n].
\end{equation}
An expression for $P_*$ as a polynomial in $Y$ and $n$ is given in \cref{sec:poly}.

We have the following result concerning roots of $P_*$:
\begin{prop}\label{prop:CurveFiniteness}
	The curve $C \subset \mathbb{A}^2$ defined by $C\colon P_* = 0$ admits only finitely many rational points. (Here we view $Y$ and $n$ as variables.)
\end{prop}
\begin{proof}
	Using a computer algebra system such as Magma, one finds that the genus of the (normalisation of) the projective closure of $C$ is $7$.
	The result then follows by Faltings' theorem (see Theorem~11.1.1 on page~352 in~\cite{bombieri2006}, especially the paragraph following it, concerning relaxation of the assumptions of the theorem as stated there).
\end{proof}

\section{Proof of \texorpdfstring{\cref{thm:Finiteness}}{Theorem\autoref{thm:Finiteness}}}
\label{sec:ProofOfTheorem}

We may assume that $n\geq 8$.
Suppose that the Galois group $G$ of $P_{n,6}$ is \emph{not} isomorphic to $\sym_6$.
Then, by \cref{thm:FM}, we must have $G \cong \PGL(2;5)$.
We shall argue below that $n$ can be completed to a point $(y,n)\in\ZZ^2$ that belongs to the curve in \cref{prop:CurveFiniteness}.
As the number of such points is finite by that proposition, the conclusion of \cref{thm:Finiteness} follows.

Consider the polynomial $P_{n,6}^{\text{rec}}$ from~\cref{eq:ReciprocalPolynomial}.
Since the roots of $P_{n,6}^{\text{rec}}$ in the algebraic closure of $\QQ$ are simply the reciprocals of the roots of $P_{n,6}$, they generate the same field extension, and the Galois groups of $P_{n,6}^{\text{rec}}$ and $P_{n,6}$ coincide.
\cref{lem:resolvent} guarantees that the polynomial $P_*$ constructed in \cref{eq:Resolvent} admits an integer root $y$ (in $Y$): $P^*(y,n) = 0$.

\section*{Acknowledgements}
The first author acknowledges support of the \emph{Austrian Science Fund (FWF)}, project number~W1230.
The second author is partially supported by the joint FWF--ANR project \emph{ArithRand} (FWF~I~4945-N and ANR-20-CE91-0006).
The second author would like to thank the algebra group at the University of Würzburg for their hospitality and helpful conversations.
Both authors are grateful to Michael Filaseta for bringing the reference~\cite{juillerat2021} to their attention, and to the anonymous referee for a careful reading of a preliminary version of the present work.

\appendix

\section{The polynomial \texorpdfstring{$P_*$}{P*}}\label{sec:poly}

The polynomial $P_*$ from~\cref{eq:Resolvent} has the form
\[
	P_* = c_*c_0 + c_*c_1 Y - c_*c_2 Y^2 - c_*c_3 Y^3 + c_*c_4 Y^4 - c_*c_5 Y^5 + Y^6 \in \QQ[Y, n],
\]
where $c_* = 2^{-49} 3^{-28} 5^{-14}$ and the coefficients $c_0,\ldots,c_5$ are polynomials in $n$, given below:
\begingroup%
\allowdisplaybreaks%
\footnotesize%
\newcommand{\brkLN}[1]{ +{} \hfill \\ #1 }%
\newcommand{\typesetCoefficient}[2]{
	#1 \times{} \\ & \quad \times
	\lparen\begin{multlined}[t]
		#2 \rparen, \hfill
	\end{multlined}
}%
\begin{align*}
	c_0 &= \typesetCoefficient{ (n-8) (n-5)^6 (n-4)^{12} (n-3)^{12} (n-2)^{13} (n-1)^{14} n^{16} }{
	23211360000 - 853011648000 n + 2333434003200 n^2 - 2935936437120 n^3 \brkLN{+} 2207949294816 n^4 - 1094840099712 n^5 + 371933306128 n^6 \brkLN{-} 87035026584 n^7 + 13698418819 n^8 - 1367062354 n^9 + 78907414 n^{10} \brkLN{-} 3408634 n^{11} + 446764 n^{12} - 58294 n^{13} + 3658 n^{14} - 102 n^{15} + n^{16} } \\
	c_1 &= \typesetCoefficient{ 2^9 3^5 5^2 (n-5)^5 (n-4)^{10} (n-3)^{10} (n-2)^{11} (n-1)^{11} n^{13} }{
	933120000 + 114124464000 n - 349180113600 n^2 + 476362740960 n^3 \brkLN{-} 390376813968 n^4 + 213042791928 n^5 - 80241328004 n^6 + 20941057250 n^7 \brkLN{-} 3686089927 n^8 + 406799205 n^9 - 22548239 n^{10} - 172271 n^{11} + 96915 n^{12} \brkLN{-} 4273 n^{13} + 23 n^{14} + n^{15} } \\
	c_2 &= \typesetCoefficient{ 2^{17} 3^{10} 5^5 (n-5)^4 (n-4)^8 (n-3)^8 (n-2)^8 (n-1)^9 n^{11} }{
	-231336000 + 594475200 n - 824462640 n^2 + 774086400 n^3 - 490731756 n^4 \brkLN{+} 207658964 n^5 - 58291061 n^6 + 10455504 n^7 - 1113065 n^8 + 59036 n^9 \brkLN{-} 519 n^{10} - 64 n^{11} + n^{12} } \\
	c_3 &= \typesetCoefficient{ 2^{27} 3^{15} 5^8 (n-5)^3 (n-4)^6 (n-3)^6 (n-2)^6 (n-1)^7 n^8 }{
	-237600 + 490800 n - 378024 n^2 + 139240 n^3 - 20288 n^4 - 3645 n^5 \brkLN{+} 1723 n^6 - 215 n^7 + 9 n^8 } \\
	c_4 &= \typesetCoefficient{ 2^{33} 3^{20} 5^{10} (n-5)^2 (n-4)^4 (n-3)^4 (n-2)^4 (n-1)^5 n^5 }{
	-7200 + 8940 n - 4888 n^2 + 1779 n^3 - 233 n^4 + n^5 + n^6 } \\
	c_5 &= 2^{42} 3^{24} 5^{12} (n-5) (n-4)^2 (n-3)^2 (n-2)^2 (n-1)^3 n^3 (n^2-5n-18).
\end{align*}%
\endgroup%

\begin{remappendix}
	On specialising $n$ to $10$, we find that
	\begingroup%
	\allowdisplaybreaks%
	\newcommand{\brkLN}[1]{ +{} \hfill \\ #1 }%
	\[ P_*(Y,10) = Y^6 \begin{multlined}[t]
		- 50803200 Y^5
		+ 996987398400000 Y^4
		\brkLN{-} 9656304644044800000000 Y^3
		\brkLN{+} 49125670785368303616000000000 Y^2
		\brkLN{-} 124886101722949886807900160000000000 Y
		\brkLN{+} 123065660595497826223597289472000000000000,
		\hfill
	\end{multlined} \]
	\endgroup%
	which has an integer root at $14817600$.
	This corresponds to the fact that the Galois group of $P_{10,6}$ is isomorphic to $\PGL(2;5)$ (see the remarks after \cref{thm:FM}).
\end{remappendix}

\begin{remappendix}
	The anonymous referee observed that one may pass to simpler curves than the one defined by the vanishing of $P_*$.
	For instance, upon replacing $Y$ by
	\[
		(n-5) (n-4)^2 (n-3)^2 (n-2)^3 (n-1)^3 n^3 Z
	\]
	and pulling out a factor of
	\(
		(n-5)^6 (n-4)^{12} (n-3)^{12} (n-2)^{13} (n-1)^{14} n^{16}
	\),
	one is left with a curve defined by the vanishing of a polynomial of the shape
	\[
		Q_{17} + Q_{16} Z + Q_{15} Z^2 + Q_{13} Z^3 + \tilde{Q}_{13} Z^4 + Q_{12} Z^5 + Q_{11} Z^6
		\in \QQ[Z, n],
	\]
	where $Q_r,\tilde{Q}_r\in\QQ[n]$ denote certain polynomials of degree $r$.
\end{remappendix}

\vfill%
\end{document}